\documentclass{amsart}%
\usepackage[T1]{fontenc}
\usepackage[latin9]{inputenc}
\usepackage{amstext}
\usepackage{amsthm}
\usepackage{amssymb}
\usepackage{amsmath}
\usepackage{amsfonts}
\usepackage{graphicx}
\usepackage{hyperref}%
\newtheoremstyle{customremark}
{3pt}
{3pt}
{}
{}
{\bfseries}
{.}
{1ex}
{}
\theoremstyle{plain}
\newtheorem{thm}{Theorem}
\newtheorem{qthm}{Theorem}
\newtheorem{lem}{Lemma}
\newtheorem{cor}{Corollary}
\newtheorem{prop}{Proposition}
\theoremstyle{customremark}
\newtheorem{rem}{Remark}
\theoremstyle{remark}
\newtheorem*{ack}{Acknowledgement}

\newcommand{\pth}{}
\input{\pth Macros.tex}
\newcommand{\nmc}{\mathcal}
\df\eqd{\overset{\mathrm{def}}{=}}
\df\card{\mathrm{card}}
\df\bds{\nmc{K}}
\df\lgt{\ell}
\df\R{\mathbb{R}}
\df\E{\mathbb{E}}
\df\Z{\mathbb{Z}}
\df\IP{\mathbb{P}}
\df\diam{\mathrm{diam}}
\df\conv{\mathrm{conv}}
\df\rmand{\text{and}}
\df\rmor{\text{or}}
\df\inn{\mathrm{int}}
\df\bd{\partial}
\df\B{\nmc{N}}
\df\M{\nmc{M}}
\df\MS{\nmc{M}^{S}}
\df\C{\nmc{C}}
\df\Diam{\nmc{D}}
\df\sp{\nmc{L}}
\df\Bno{\widetilde{\nmc{N}}}
\df\BF{\nmc{F}}
\df\Bo{\mathbb{B}}
\df\Bc{\bar{\mathbb{B}}}
\df\IS {\mathbb{S}}
\df\uc#1{\gamma_{s}^{#1}}
\df\lc#1{\gamma_{i}^{#1}}
\df\ucr#1{\rho_{s}^{#1}}
\df\lcr#1{\rho_{i}^{#1}}
\df\TUS{T^1}

\begin{document}
\author{Alain Rivi\`{e}re}
\author{Jo\"{e}l Rouyer}
\author{Costin V\^\i lcu}
\author{Tudor Zamfirescu}
\title{Double normals of most convex bodies}
\date{\today}

\begin{abstract}
We consider a typical (in the sense of Baire categories) convex body $K$ in
$\R^{d+1}$. The set of feet of its double normals is a Cantor set, having
lower box-counting dimension $0$ and packing dimension $d$. The set of lengths
of those double normals is also a Cantor set of lower box-counting dimension
$0$. Its packing dimension is equal to $\frac{1}{2}$ if $d=1$, is at least
$\frac{3}{4}$ if $d=2$, and equals $1$ if $d\geq3$. We also consider the lower
and upper curvatures at feet of double normals of $K$, with a special interest
for local maxima of the length function (they are countable and dense in the
set of double normals). In particular, we improve a previous result about the
metric diameter.

\end{abstract}
\maketitle




\section{Introduction and results}

Let $\E$ be the Euclidean space of dimension $d+1$, with $d\geq1$, and let
$\bds$ be the set of all convex bodies (\ie, compact convex sets with
non-empty interior) in $\E$. For $K\in\bds$, a \emph{chord} is a line segment
$xy$ joining boundary points $x$ and $y$ of $K$. A chord $xy$ is called a
\emph{normal} of $K$ if it is orthogonal to some supporting hyperplane at the
point $x$ called \emph{foot}. An \textit{affine diameter} is a chord with
parallel supporting hyperplanes at its endpoints, while a \textit{double
normal} is an affine diameter orthogonal to those supporting hyperplanes.
Thus, a double normal is a normal with two feet. In this paper, $\B(K)$ stands
for the set of (oriented) double normals of $K$, $\lgt(c)$ denotes the length
of an oriented chord $c$, and $\sp(K)=\set(:\lgt(b)|b\in\B(K):)$.

It is well-known that every normal to a convex body $K$ is a double normal if
and only if $K$ has constant width. On the other hand, the shortest and the
longest affine diameter are double normals, but are there others?

Answering a question proposed by V. Klee \cite{Klee}, N. N. Kuiper proved in
1964 that every convex body in $\E$ has at least $d+1$ non-oriented double
normals \cite{Kui}. Moreover, for any $\mathcal{C}^{2-}$-function
$f:\IP^{d}\rightarrow\R$ there exists a symmetric convex body $K$ in $\E$ with
centre $0$, for which the set of double normals directions coincides to the
critical set $\set(:z\in\IP^{d}|(df)_{z}=0:)$ of $f$. Conversely, for any
convex body $K$ in $\E$ there exists a centrally symmetric convex body
$K^{\prime}$ with $\mathcal{C}^{2-}$-boundary and a $\mathcal{C}^{2-}%
$-function $f:\IP^{d}\rightarrow\R$ whose critical set coincides to the set
double normal directions of $K$, and of $K^{\prime}$. Here $\mathcal{C}^{2-}$
stands for a class of regularity between $\mathcal{C}^{1}$ and $\mathcal{C}%
^{2}$. More important for our paper, he also proved the following result.

\begin{qthm}
\label{qthm:Kuiper}\cite{Kui} If $d\leq2$, $\sp(K)$ has measure $0$, while for
$d\geq3$ there exists a $\mathcal{C}^{2-}$ centrally symmetric strictly convex
body $K^{\star}$ in $\E$ and a (non-rectifiable) arc $\gamma:[0,1]\rightarrow
\B(K^{\star})$ such that $\sp(K^{\star})=\set(:\lgt(\gamma\left(  t\right)
)|t\in\lbrack0,1]:)$ is a non-degenerate interval.
\end{qthm}

Two years later, A. S. Besicovitch and T. Zamfirescu \cite{BZ} proved the
existence of a planar convex body $K$ with an interior point $x$ such that
$\sp\left(  K\right)  $ and the set of ratios in which $x$ divides affine
diameters through it are uncountable. Their construction provides convex
curves whose set of double normals is homeomorphic to any chosen compact
subset of $\R$.

Recently, J. P. Moreno and A. Seeger devoted Sections 4 and 5 in \cite{MoSe}
to the study of double normals. They prove, among other results, that $\sp(K)$
is finite for any full dimensional polytope $K$ in $\E$ (compare to our Lemma
\ref{lem:std-finite}). \medskip{}

Kuiper's results are closely related to billiards. Indeed, on a convex
billiard table, $2$-periodic trajectories correspond to double normals. A
classical result of G. Birkhoff \cite{bf} states that in any planar billiard
table $K$ there always exist trajectories of period $n$, for any integer
$n\geq2$.

The set $\mathcal{B}$ of strictly convex planar sets, having a $\mathcal{C}%
^{r}$ boundary (for some $r\geq2$) with positive curvature everywhere, endowed
with a suitable metric, is a Baire space.

M. J. Dias Carneiro, S. Olfison Kamphorst and S. Pinto de Carvalho \cite{DOP}
proved that for most billiard tables $K\in\mathcal{B}$, for every integer
$n\geq2$, there are at most finitely many $n$-periodic trajectories; in
particular, $\B\left(  K\right)  $ and thus $\sp(K)$ are finite.\textsl{ }For
results in similar directions, see \cite{BMR}, \cite{DOP-I}, \cite{DOP},
\cite{KoC}, \cite{Koz}, \cite{Koz-d}, \cite{PeSto}, \cite{XZ}.

\medskip The problem of counting double normals extends beyond convexity, to
the framework of Riemannian manifolds, see for instance \cite{Hay},
\cite{Rie}, \cite{TaW}.

\medskip

In this paper we study double normals from the point of view of Baire
categories. Our results strongly contrast the abovementioned ones on the
finiteness of the sets of double normals.

\medskip

The next fundamental fact, independently discovered by V. Klee \cite{Klee-B}
and P. Gruber \cite{Grub-B}, is essential for our topic.

\begin{qthm}
\label{qthm:Klee-Gruber}\cit{Grub-B}, \cit{Klee-B} The boundary of most
$K\in\bds$ is of differentiability class $\mathcal{C}^{1}\setminus
\mathcal{C}^{2}$ and strictly convex.
\end{qthm}

Our work is also related to the articles \cite{BaZ}, \cite{Zamfi-mirr},
\cite{Zamfi-norm}, \cite{TZ-diam}, which focus on intersections of infinitely
many affine diameters or normals for typical convex bodies. Let us mention
here that, for $d\geq2$, double normals of a typical convex body are pairwise
disjoint \cite{RV}. For other Baire categories results about convex bodies,
see e.g. the survey \cite{Zamfi-surv}.

\medskip

We prove in this paper the following results.

For most $K\in\bds$, the set of feet of double normals is a Cantor set (\ie, a
set homeomorphic to the standard Cantor set) having lower box-counting
dimension $0$ and packing dimension $d$ (Theorem \ref{thm:cantor} in Section
\ref{sec:Cantor}, and Theorems \ref{T5}--\ref{thm:UPD-bin} in Section
\ref{sec:dim}). Recall that the lower box-counting dimension is greater than
or equal to the Hausdorff dimension and the upper box-counting dimension is
greater than or equal to the packing dimension, so these results provide the
typical Hausdorff and the upper box-counting dimension as well.

Double normals are related to the critical points of $\lgt _{K}$, see Lemmas
\ref{lem:locmax} and \ref{lem:bin-crit}.

For most $K\in\bds$, the map $\lgt_{K}$, which associates to a double normal
of $K$ its length, is injective and Lipschitz continuous. Thus, $\sp(K)$ is a
Cantor set and has lower box-counting dimension $0$. In particular, its
Lebesgue measure vanishes, though the function $\lgt_{K}$ does not satisfy the
hypotheses of regularity of Sard's theorem. For most $K\in\bds$, the packing
dimension of $\sp(K)$ is equal to $\frac{1}{2}$ if $d=1$, is at least
$\frac{3}{4}$ if $d=2$, and equals $1$ if $d\geq3$ (Theorems \ref{thm:inj}%
--\ref{thm:UPD-values} and Corollary \ref{cor:LPDvalue} in Section
\ref{sec:crit-val}).

Again, for most $K\in\bds$, the set of maximizing chords (local maxima of the
length function) is countable and dense in $\B(K)$ (Propositions
\ref{prp:countable}--\ref{prp:dense} in Section \ref{sec:KoCP}).

The last author considered in \cite{Zamfi-curv I}, \cite{Zamfi-curv II},
\cite{Zamfi-curv III} the lower and upper curvatures $\lc\tau$ and $\uc\tau$
and proved, among other results, the following fact.

\begin{qthm}
\label{qthm:TZ}For most $K\in\bds$, at each point $x\in\bd K$, $\lc\tau(x)=0$
or $\uc\tau(x)=\infty$ for any tangent direction $\tau$ at $x$, and both
equalities hold at most points.
\end{qthm}

The curvature of a convex body is deeply related to double normals, see
\cite{BS}, \cite{Zamfi-cw} and Remark \ref{rmk:index}.

We continue this investigation by considering the lower and upper curvatures
at feet of double normals. We prove that at any foot $x$ of a maximizing chord
$c$ of a typical convex body and in any tangent direction $\tau$,
$\uc\tau(x)=\infty$ and $\lc\tau(x)\geq{\lgt(c)}^{-1}$, with equality if $c$
is a metric diameter (a chord of globally maximal length); this improves
\cite[Th. 11]{Zamfi-diam}. Moreover, at both feet of a typical double normal,
$\uc\tau(x)=\infty$. Finally, given a fixed line-segment $c=xy$, for most
convex bodies admitting $c$ as double normal, $\uc\tau(x)=\infty$ and
$\lc\tau(x)=0$ (Theorems \ref{thm:most-max-chord}--\ref{thm:fixed-d-normal} in
Section \ref{sec:curv-feet}).

Statements similar to our theorems, but involving only centrally-symmetric
convex bodies in $\E$, can also be proven. In this case, due to a variant of
Theorem \ref{qthm:Klee-Gruber} for these bodies, see also \cite[Theorem
2]{Kui}, all double normals intersect at the symmetry centre. The formal
statements and the proofs are left to the interested reader. This paper also
leaves open several questions, see
Remarks \ref{Rmk:diam}, \ref{Rmk:34} and \ref{rmk:cfdn}.

\section{Preliminaries}

The space $\bds$, endowed with the Pompeiu-Hausdorff metric $d_{PH}$, is a
Baire space. This allows us to state that \emph{most} convex bodies, or
\emph{typical} convex bodies enjoy a given property, meaning that the set of
those bodies that do not enjoy it is meagre, \ie of first Baire category.
(Recall that a subset of a topological space is said to be of \emph{first
Baire category}, if it is included in a countable union of closed sets of
empty interior. Otherwise, it is called of second category.) Of course, it is
also equivalent to state that the set of bodies that do enjoy the considered
property is \emph{residual}, meaning that it contains a dense countable
intersection of open sets (a dense $G_{\delta}$-set). We shall need the
following (almost obvious) lemma.

\begin{lem}
\label{lem:Z}\cit{AZ1} If $Z$ is a space of second Baire category (in itself),
$Y$ is residual in $Z$, and $X$ is residual in $Y$, then $X$ is residual in
$Z$.
\end{lem}

In this article, we shall apply the lemma when $Z$ is Baire space.

The distance on the set $\C\left(  K\right)  \eqd\bd K\times\bd K$ of
(possibly degenerated) oriented chords of $K$ is%

\[
d\left(  \left(  x,y\right)  ,\left(  x^{\prime},y^{\prime}\right)  \right)
=\max\left(  \left\Vert x-x^{\prime}\right\Vert ,\left\Vert y-y^{\prime
}\right\Vert \right)  ;
\]
thus, the ball centred at $c\in\C\left(  K\right)  $ of radius $r$ coincides
with the Cartesian product of the balls of radii $r$ centred at the
extremities of $c$. The restriction of $\lgt$ to $\C\left(  K\right)  $ is
denoted by $\lgt_{K}$. An oriented chord which is a local maximum (a strict
local maximum) of $\lgt_{K}$ is said to be \emph{maximizing }%
(respectively\emph{ strictly maximizing}). We define $\M(K)$ (resp. $\MS(K)$)
as the set of maximizing chords (respectively strictly maximizing chords).

From now on, unless otherwise specified, the words \emph{double normal} will
refer to an oriented double normal. The set of non-oriented double normals of
$K$ is denoted by $\Bno\left(  K\right)  $; and $\widetilde{\ell}_{K}$ stands
for the corresponding length map. It's easy to see that the canonical map
$\phi_{K}:\B\left(  K\right)  \rightarrow\Bno\left(  K\right)  $ is
$1$-Lipschitz. The set of feet of double normals is denoted by $\BF\left(
K\right)  $. The set of affine diameters of $K$ is denoted by $\Diam\left(
K\right)  $.

Some more general notation follows. We denote by $\mathbb{N}_{n}$ the set of
positive integers smaller than or equal to $n$ and by $\mathbb{N}_{n}^{0}$ the
set of non-negative integers smaller than $n$. Given an $n$-tuple $x=\left(
x_{1},\ldots,x_{n}\right)  \in\E^{n}$ and a subset $I$ of $\mathbb{N}_{n}$,
$x_{I}$ denotes the set $\set(:x_{i}|i\in I:)$.

For any subset $A$ of $\E$, $\bd A$ stands for the boundary of $A$, $\conv(A)$
for the convex hull of $A$ (\ie, the intersection of all convex sets
containing $A$), $\left\langle A\right\rangle $ for the affine space spanned
by $A$ and $\overrightarrow{A}$ for the direction of $\left\langle
A\right\rangle $.

For distinct $x$, $y\in\E$, $xy$ stands for the line segment joining $x$ to
$y$ and $\overline{xy}$ for the whole line. The open ball, closed ball and
sphere centred at $x$ of radius $r$ are denoted by $\Bo\left(  x,r\right)  $,
$\Bc\left(  x,r\right)  $ and $\IS\left(  x,r\right)  $ respectively.

The following lemma is obvious and left to the reader.

\begin{lem}
\label{lem:USC} Let $K_{n}\in\bds$ tend to $K\in\bds$. Let $\left(
x_{n},y_{n}\right)  \in\B\left(  K_{n}\right)  $ converge to $\left(
x,y\right)  \in\E^{2}$. Then $\left(  x,y\right)  $ is a double normal of $K$.
\end{lem}

Double normals are related to the critical points of $\lgt _{K}$. More
precisely we have the following two lemmas.

\begin{lem}
\label{lem:locmax}If $b=\left(  x,y\right)  $ is a local maximum of
$\lgt _{K}$, then $b$ is a double normal.
\end{lem}

\begin{proof}
Assume that $b$ is not a double normal. Then the hyperplane $H$ normal to
$\overline{xy}$ through one foot of $b$, say $x$, is not a supporting
hyperplane. It follows that there exists $x_{n}\in\bd K$ tending to $x$ and
separated from $y$ by $H$. Thus, $\left\Vert y-x_{n}\right\Vert >\left\Vert
x-y\right\Vert $ and $\left(  x,y\right)  $ is not a local maximum of
$\lgt_{K}$.
\end{proof}

The next lemma is Proposition 1 in \cite{Koz-d}; see also Proposition 2.2 in
\cite{DOP-I}.

\begin{lem}
\label{lem:bin-crit}If $\bd K$ is $\mathcal{C}^{1}$ then $b\in\C\left(
K\right)  $ is a double normal if and only if $\lgt\left(  b\right)  >0$ and
$\left(  d\lgt_{K}\right)  _{b}=0$.
\end{lem}

The next lemma is central to this paper.

\begin{lem}
\label{lem:stab}Let $b\in\MS\left(  K\right)  $. Then, for any $\varepsilon
>0$, there exists a neighbourhood $\mathcal{U}$ of $K$ in $\bds$ such that for
any $K^{\prime}\in\mathcal{U}$ there exists a maximizing chord $b^{\prime}%
\in\M\left(  K^{\prime}\right)  $ satisfying $d\left(  b,b^{\prime}\right)
<\varepsilon$.
\end{lem}

\begin{proof}
Since $b$ is a strict local maximum of $\lgt$, there exists $r\in\left]
0,\min\left(  \varepsilon,\lgt\left(  b\right)  \right)  \right[  $ such that
\[
\lgt\left(  b\right)  >\max_{c\in\IS\left(  b,r\right)  \cap\C\left(
K\right)  }\lgt\left(  c\right)  \text{.}%
\]
Hence, there is a neighbourhood $\mathcal{U}$ of $K$ such that for any
$K^{\prime}\in\mathcal{U}$ there exists $c^{\prime}\in\C\left(  K^{\prime
}\right)  \cap\Bo\left(  b,r\right)  $ verifying%

\[
\lgt\left(  c^{\prime}\right)  >\max_{c\in\IS\left(  b,r\right)  \cap\C\left(
K^{\prime}\right)  }\lgt\left(  c\right)  \text{.}%
\]

It follows that the global maximum $b^{\prime}$ of $\lgt_{K^{\prime}}$ on
$\Bc\left(  b,r\right)  \cap\C\left(  K^{\prime}\right)  $ is not achieved on
the boundary of the ball, and thus, it is a maximizing chord.
\end{proof}

We will often use implicitly the following criterion in order to prove that a
chord is maximizing.

\begin{lem}
\label{lem:strict-max-criterion}Let $(x,y)\in\B(K)$, $K\in\bds$. If there
exists $\eta>0$ and $\alpha<\pi/2$ such that for any $\left(  x^{\prime
},y^{\prime}\right)  \in\Bo(x,\eta)\times\Bo(y,\eta)$ the angles
$\measuredangle yxx^{\prime}$ and $\measuredangle xyy^{\prime}$ are smaller
than $\alpha$, then $(x,y)\in\MS(K)$.
\end{lem}

The proof is elementary and left to the reader.

\begin{cor}
\label{cor:polyMMS}For any polytope $K\in\bds$, $\MS(K)=\M(K)$.
\end{cor}

The next lemma will be invoked in the proof of Theorem \ref{thm:UPD-values}.
It seems to be interesting by itself.

\begin{lem}
\label{lem:Holder}For all $K\in\bds$, the map $\lgt_{K}$ is $2$-H\"{o}lder
continuous. More precisely, for any $b$$_{0}$, $b_{1}\in\B(K)$, $\left\vert
\lgt_{K}(b_{0})-\lgt_{K}(b_{1})\right\vert \leq H_{K}d(b_{0},b_{1})^{2}$,
where $H_{K}=1/\min\sp\left(  K\right)  $.
\end{lem}

\begin{proof}
Assume that $\lgt(b_{0})\leq\lgt(b_{1})$ and set $\varepsilon\eqd d\left(
b_{0},b_{1}\right)  $. Let $x$, $x^{\prime}$ be the feet of $b_{0}$ and $x+u$,
$x^{\prime}+u^{\prime}$ be the feet of $b_{1}$, where $\max(\left\Vert
u\right\Vert ,\left\Vert u^{\prime}\right\Vert )\leq\varepsilon$. Since
$b_{1}$ is included in the zone of $\E$ between the hyperplanes normal to
$b_{0}$ through $x$ and $x^{\prime}$, we have $\left\langle x^{\prime
}-x,u\right\rangle \geq0$ and $\left\langle x^{\prime}-x,u^{\prime
}\right\rangle \leq0$. It follows that
\begin{align*}
\lgt(b_{1})^{2}  &  =\left\Vert x^{\prime}-x+u^{\prime}-u\right\Vert ^{2}\\
&  =\lgt(b_{0})^{2}+\left\Vert u\right\Vert ^{2}+\left\Vert u^{\prime
}\right\Vert ^{2}-\left\langle x^{\prime}-x,u\right\rangle +\left\langle
x^{\prime}-x,u^{\prime}\right\rangle \\
&  \leq\lgt(b_{0})^{2}+2\varepsilon^{2}\text{,}%
\end{align*}
whence
\[
\lgt(b_{1})-\lgt(b_{0})\leq\frac{2}{\lgt(b_{1})+\lgt(b_{0})}\varepsilon
^{2}\leq H_{K}\varepsilon^{2}\text{.}%
\]

\end{proof}

\begin{rem}
$2$-H\"{o}lder maps defined on a space connected by Lipschitz continuous arcs
are constant.
\end{rem}

\begin{rem}
It is classical that the restriction of a map of class $\mathcal{C}^{2}$ to a
compact set of critical points is always $2$-H\"{o}lder, but in our case
$\lgt_{K}$ is not so regular.
\end{rem}


\section{A Cantor set\label{sec:Cantor}}

In this section, we prove the following theorem.

\begin{thm}
\label{thm:cantor}For most $K\in\bds$, $\B(K)$ is homeomorphic to the Cantor set.
\end{thm}

\begin{proof}
Recall that a famous theorem of Brouwer assures that a compact metric space is
a Cantor set if and only if it is non-empty, totally disconnected, and
perfect. The compactness is clear from Lemma \ref{lem:USC}. The non-emptiness
follows from the fact that any metric diameter (\ie, longest chord) is, by
Lemma \ref{lem:locmax}, a double normal. Thus, it remains to prove the last
two properties, to which Lemmas \ref{lem:tot-disc} and \ref{lem:perfect} below
are devoted.
\end{proof}

A finite set $X\subset\E$ is said to be \emph{standard} if for any two
disjoint subsets $X_{1}$, $X_{2}$ with cardinality at most $d+1$, we have
\[
\dim\left(  \overrightarrow{X_{1}}\cap\overrightarrow{X_{2}}\right)
=\max\left(  0,\dim\left(  \overrightarrow{X_{1}}\right)  +\dim\left(
\overrightarrow{X_{2}}\right)  -d-1\right)  \text{.}%
\]
A polytope is said to be \emph{standard} if for any two facets $F$, $G$ that
does not have a common vertex, we have
\[
\dim\left(  \overrightarrow{F}\cap\overrightarrow{G}\right)  =\max\left(
0,\dim\left(  \overrightarrow{F}\right)  +\dim\left(  \overrightarrow
{G}\right)  -d-1\right)  \text{.}%
\]
Clearly, a polytope with a standard set of vertices is standard.

\begin{lem}
\label{lem:std-finite}If $K\in\bds$ is a standard polytope then $\B\left(
K\right)  $ is finite.
\end{lem}

\begin{proof}
Let $\left(  x,y\right)  \in\B\left(  K\right)  $ and $F_{x}$, $F_{y}$ be the
minimal dimensional facets containing $x$ and $y$ respectively. Clearly
$F_{x}$ and $F_{y}$ are included in two parallel supporting hyperplanes
$H_{x}$ and $H_{y}$, whence they cannot have a common vertex. On the one hand,
$K$ is standard, whence
\[
\dim\left(  \overrightarrow{F_{x}}\cap\overrightarrow{F_{y}}\right)
=\max\left(  0,\dim\left(  \overrightarrow{F_{x}}\right)  +\dim\left(
\overrightarrow{F_{y}}\right)  -d-1\right)  \text{.}%
\]
On the other hand, $\overrightarrow{F_{x}}$ and $\overrightarrow{F_{y}}$ are
subspaces of $\overrightarrow{H_{x}}=\overrightarrow{H_{y}}$, whence
\[
\dim\left(  \overrightarrow{F_{x}}\cap\overrightarrow{F_{y}}\right)  \geq
\max\left(  0,\dim\left(  \overrightarrow{F_{x}}\right)  +\dim\left(
\overrightarrow{F_{y}}\right)  -d\right)  \text{.}%
\]
It follows that $\dim\left(  \overrightarrow{F_{x}}\right)  +\dim\left(
\overrightarrow{F_{y}}\right)  \leq d$ and $\dim\left(  \overrightarrow{F_{x}%
}\cap\overrightarrow{F_{y}}\right)  =0$. Hence $\left(  x,y\right)  $ is the
only double normal whose extremities lie in minimal facets $F_{x}$ and $F_{y}%
$. We proved that, for any pair of facets, there is at most one double normal.
It follows that $\B\left(  K\right)  $ is finite.
\end{proof}

\begin{lem}
\label{lem:std-dense}The set of $n$-tuples $x\in\E^{n}$ such that
$x_{\mathbb{N}_{n}}$ is standard contains an open and dense set in $\E^{n}$.
\end{lem}

\begin{proof}
First notice that the set $U\subset\E^{n}$ of all $n$-tuples $x$ such that for
any $I\subset\mathbb{N}_{n}$, $\dim\overrightarrow{x_{I}}=\min\left(
\#I-1,d+1\right)  $ (points in \emph{generic} position) is open and dense. We
have to prove that, for any non-empty disjoints subsets $I$, $J$ with
cardinality at most $d+1$, the set
\[
U_{I,J}\eqd\set(:x\in U|\dim\left(  \overrightarrow{x_{I}}\cap\overrightarrow
{x_{J}}\right)  =\max\left(  0,\#I+\#J-3-d\right)  :)
\]
is open and dense. Put $k\eqd\max\left(  0,\#I+\#J-d-3\right)  $. Note that
$\dim\overrightarrow{x_{I}}\cap\overrightarrow{x_{J}}$ is always greater than
or equal to $k$, and that $\mathrm{rank}\left(  M_{IJ}\right)  =\#I+\#J-2-\dim
\left(  \overrightarrow{x_{I}}\cap\overrightarrow{x_{J}}\right)  $, where
$M_{IJ}$ is a $\left(  d+1\right)  \times(\#I+\#J-2)$ matrix, whose columns
are vectors $x_{i}-x_{\min I}$ ($i\in I$, $i\neq\min I$) and $y_{j}-y_{\min
J}$ ($j\in J$, $j\neq\min J$). So $x\notin U_{IJ}$ if and only if
$\mathrm{rank}\left(  M_{IJ}\right)  <\#I+\#J-2-k$, that is, if all minors of
$M_{IJ}$ of order greater than or equal to $\#I+\#J-2-k$ vanish. Such minors
are polynomials on $\E^{n}$, whence $U_{IJ}$ is open, and dense if and only if
it is not empty. The latter fact being obvious, the proof is finished.
\end{proof}

\begin{lem}
\label{lem:tot-disc}For most $K\in\bds$, $\B\left(  K\right)  $ is totally disconnected.
\end{lem}

\begin{proof}
We have
\begin{align*}
\mathcal{A}=  &  \set(:K\in\bds|\exists C\subset\B\left(  K\right)  ,\,C\text{
is connected},\,\diam\left(  C\right)  >0:)\\
=  &  \bigcup_{n\in\mathbb{N}^{*}}\set(:K\in\bds|\exists C\subset\B\left(
K\right)  ,\,C\text{ is connected},\,\diam\left(  C\right)  \geq\frac{1}%
{n}:)\\
\eqd  &  \bigcup_{n\in\mathbb{N}^{*}}\mathcal{A}_{n}\text{.}%
\end{align*}

It is well-known that a limit of connected sets is connected. Hence, by Lemma
\ref{lem:USC}, $\mathcal{A}_{n}$ is closed. By virtue of Lemma
\ref{lem:std-dense}, standard polytopes are dense in $\bds$ , and by Lemma
\ref{lem:std-finite}, they cannot belong to $\mathcal{A}_{n}$. Hence
$\mathcal{A}_{n}$ has empty interior, and thus $\mathcal{A}$ is meagre.
\end{proof}

\begin{lem}
\label{lem:perfect}For most $K\in\bds$, $\B\left(  K\right)  $ is perfect.
\end{lem}

\begin{proof}
Chose any countable dense set $Z$ in $\E^{2}$. The assumption that $\B\left(
K\right)  $ is not perfect implies that there exit $b\in\B\left(  K\right)  $,
$r>0$, $u\in Z$ such that
\[
\B\left(  K\right)  \cap\Bc\left(  u,r\right)  =\left\{  b\right\}  \text{.}%
\]

We have
\[
\mathcal{A}\eqd\set(:K\in\bds|\B\left(  K\right)  \,\text{not perfect}%
:)\subset\bigcup_{\left(  n,u\right)  \in\mathbb{N}^{\ast}\times Z}%
\mathcal{A}_{n,u}%
\]
with%

\begin{align*}
\mathcal{A}_{n,u}=  &  \set(:K\in\bds|\exists b\in\B\left(  K\right)
\,\text{s.t. }\B\left(  K\right)  \cap\Bc\left(  u,\frac{1}{n}\right)
=\left\{  b\right\}  :)\\
=  &  \set(:K\in\bds|\#\left(  \B\left(  K\right)  \cap\Bc\left(  u,\frac
{1}{n}\right)  \right)  =1:)
\end{align*}
We have to prove that the closure of $\mathcal{A}_{n,u}$ has empty interior,
that is, for any $K_{0}\in\bds$ and any $\varepsilon>0$ there exists $K_{3}%
\in\bds$ such that $d_{PH}\left(  K,K_{3}\right)  <\varepsilon$ and such that
a whole neighbourhood of $K_{3}$ does not intersect $\mathcal{A}_{n,u}$.

First, we can find a polytope $K_{1}$ such that $d_{PH}\left(  K_{0}%
,K_{1}\right)  <\varepsilon$. If $\B\left(  K_{1}\right)  \cap\Bc\left(
u,\frac{1}{n}\right)  $ is empty then the set will remain empty for any $K$ in
a whole neighbourhood of $K_{1}$, because otherwise the limit of a converging
subsequence of double normals of $K$ tending to $K_{1}$ would belong to
$\Bc\left(  u,\frac{1}{n}\right)  $. Hence we can set $K_{3}=K_{1}$ and the
proof is finished.

If $\B\left(  K_{1}\right)  \cap\Bc\left(  u,\frac{1}{n}\right)  $ is not
empty then we can move and dilate slightly $K_{1}$ such that the modified
polytope $K_{2}$ satisfies $d_{PH}\left(  K_{0},K_{2}\right)  <\varepsilon$
and $\B\left(  K_{2}\right)  \cap\Bo\left(  u,\frac{1}{n}\right)
\neq\emptyset$. Let $b_{2}$ belong to $\B\left(  K_{2}\right)  \cap\Bo\left(
u,\frac{1}{n}\right)  $.
Consider a rectangle $R=x_{3}x_{3}^{\prime}y_{3}y_{3}^{\prime}$ whose centre
is the midpoint of $b_{2}$, such that $x_{3}y_{3}^{\prime}$ is parallel to
$b_{2}$, longer than $\lgt(b_{2})$. If it is not too long nor too wide, then
$(x_{3},y_{3})$ and $(x_{3}^{\prime},y_{3}^{\prime})$ belong to $\Bo\left(
u,\frac{1}{n}\right)  $
and the distance from $K_{3}\eqd\conv\left(  K_{2}\cup R\right)  $ to $K_{0}$
is less than $\varepsilon$. Still reducing the width $x_{3}x_{3}^{\prime}$ if
necessary, we may assume that the hyperplanes normal to the diagonals of $R$
through their extremities does not intersect $K_{2}$, whence those hyperplanes
are supporting $K_{3}$, and $\left(  x_{3},y_{3}\right)  $ and $\left(
x_{3}^{\prime},y_{3}^{\prime}\right)  $ are double normal of $K_{3}$. Also,
one can easily check that any segment between $x_{3}$ (respectively $y_{3}$)
and any point of $K_{3}$ make an angle less than $\pi/2$, whence $\left(
x_{3},y_{3}\right)  \in\MS(K_{3})$. Of course, the same holds for $\left(
x_{3}^{\prime},\text{$y_{3}^{\prime}$}\right)  $. Now, by Lemma \ref{lem:stab}%
, there is a whole neighbourhood $U$ of $K_{3}$ such that any $K\in U$ admits
at least two double normals in $\Bo\left(  u,\frac{1}{n}\right)  $, hence $U$
does not intersect $\mathcal{A}_{n,u}$.
\end{proof}


\section{Dimensions\label{sec:dim}}

In this section, we prove that for most convex bodies $K$ the lower
box-counting dimension of $\BF\left(  K\right)  $ is $0$ and its packing
dimension is $d$. Let us recall their definitions.

If $A$ is a metric space and $\delta$ is a positive number, a subset $F\subset
A$ is called a \emph{$\delta$-set} if any two distinct points of $F$ have a
distance at least $\delta$. Let's denote by $P_{\delta}(A)$ the supremum of
the cardinals of all $\delta$-sets of $A$. The \emph{lower} and \emph{upper
box-counting dimension} of $A$ are defined as
\begin{align*}
\underline{\dim}_{B}A=  &  \liminf_{\delta\rightarrow0}\frac{\ln P_{\delta
}(A)}{-\ln\delta}\\
\overline{\dim}_{B}A=  &  \limsup_{\delta\rightarrow0}\frac{\ln P_{\delta}%
(A)}{-\ln\delta}\text{.}%
\end{align*}
It is well-known that the lower packing dimension is greater than or equal to
the Hausdorff dimension \cite{falc}.

The fact that a compact countable set may have arbitrarily large box-dimension
leads to the definition of the so-called \emph{packing dimension}:%
\[
\dim_{P}A=\inf_{\left\{  A_{i}\right\}  _{i\in\mathbb{N}}}\sup_{i\in
\mathbb{N}}\overline{\dim}_{B}A_{i}\text{,}%
\]
where the infimum is taken over all the coverings $\left\{  A_{i}\right\}
_{i\in\mathbb{N}}$ of $A$. It is clear that this dimension is lower than or
equal to the upper box dimension, and vanishes for any countable set. There
also exists a similar dimension derived from the lower box-counting dimension,
but we shall not use it in this paper. Note that, classically, the packing
dimension is defined in a completely different way, involving outer measures.
See \cite[3.3 and 3.4]{falc} or \cite[Theorem 5.9]{Mattila} for the original
definition and the equivalence between those definitions.

It is easy to see that for any subset $A$ of $\E$, $P_{\delta}\left(
\overline{A}\right)  =P_{\delta}\left(  A\right)  $ and thus $\overline{\dim
}_{B}\overline{A}=\overline{\dim}_{B}A$. This fact, together with Baire's
theorem leads to the following lemma.

\begin{lem}
\label{lem:dimP-Homo}Let $s$ be a positive number. If $A$ is a complete metric
space in which any open set has upper counting-box dimension at least $s$,
then $\dim_{P}A\geq s$.
\end{lem}

It follows that $\overline{\dim}_{B}A=\dim_{P}A$ whenever $A$ is complete and
enjoys some kind of homogeneity, as can be expected for the set of double
normals of a typical convex body.

\begin{thm}
\label{T5}For most $K\in\bds$, the lower box-counting dimension of $\B\left(
K\right)  $ is $0$.
\end{thm}

Using an general result of Gruber \cite[p. 20]{Grub2}, the proof of the
theorem almost completely reduces to the upper semi-continuity of the maps
$K\mapsto\B\left(  K\right)  $ (Lemma \ref{lem:USC}) and $A\mapsto P_{\delta
}\left(  A\right)  $ (\cite[p. 20]{Grub2}). However, in order to make the
paper more self-contained, we choose to give a more geometrical, direct proof.

\begin{proof}
Define%
\begin{align*}
\mathcal{A}  &  \overset{\mathrm{def}}{=}\set(:K\in\bds|\liminf\frac{\log
P_{\delta}\left(  \B\left(  K\right)  \right)  }{-\log\delta}>0:)\\
&  =\bigcup_{n}\set(:K\in\bds|\liminf\frac{\log P_{\delta}\left(  \B\left(
K\right)  \right)  }{-\log\delta}\geq\frac{1}{n}:)=\bigcup_{n,m}%
\mathcal{A}_{n,m}\text{,}%
\end{align*}
where%
\[
\mathcal{A}_{n,m}=\set(:K\in\mathcal{K}|\forall\delta\leq\frac{1}{m}%
~\frac{\log P_{\delta}\left(  \B\left(  K\right)  \right)  }{-\log\delta}%
\geq\frac{1}{n}:)\text{.}%
\]

We first prove that $\mathcal{A}_{n,m}$ is closed. Let $K_{p}\in
\mathcal{A}_{n,m}$ tend to $K\in\bds$. Let us fix $\delta\leq1/m$; we want to
prove that%
\[
P_{\delta}\left(  \B\left(  K\right)  \right)  \geq\delta^{-1/n}\text{.}%
\]

Since $K_{p}\in\mathcal{A}_{n,m}$ we have $P_{\delta}\left(  \B\left(
K_{p}\right)  \right)  \geq\delta^{-1/n}$. So there are $N\overset
{\mathrm{def}}{=}\left\lceil \delta^{-1/n}\right\rceil $ double normals
$b_{p}^{1}$, \ldots, $b_{p}^{N}$ in $K_{p}$ forming a $\delta$-set. By
extraction, one can assume the convergence of each sequence $\left\{
b_{p}^{i}\right\}  _{p}$ ($i\in\mathbb{N}_{N}$) to some limit $b^{i}%
\in\B\left(  K\right)  $ (by Lemma \ref{lem:USC}). Obviously $\left\{
b^{i}|i\in\mathbb{N}_{N}\right\}  $ is a $\delta$-set of double normals. So
$P_{\delta}\left(  \B\left(  K\right)  \right)  \geq\delta^{-1/n}$, and
$K\in\mathcal{A}_{n,m}$.

Clearly, if $\B\left(  K^{\prime}\right)  $ is finite for some $K^{\prime}%
\in\bds$ then $K^{\prime}$ does not belongs to $\mathcal{A}_{nm}$. Hence, by
Lemmas 7 and 8, $\mathcal{A}_{n,m}$ has empty interior and $\mathcal{A}$ is meagre.
\end{proof}

\begin{lem}
\label{lem:spher-cap}For any $K\in\bds$, any $\left(  x,y\right)  \in\B\left(
K\right)  $ and any $\varepsilon>0$ there exist $K^{\prime}\in\bds$, $o\in\E$,
$R>0$ such that $d_{PH}\left(  K,K^{\prime}\right)  <$ $\varepsilon$ and
$\IS(o,R)\cap\bd K^{\prime}$ contains two spherical caps symmetrical to each
other with respect to $o$, one of them included in $\Bo\left(  x,\varepsilon
\right)  $.
\end{lem}

\begin{proof}
Let $o$ be the midpoint of $xy$ and $\Delta$ the open subset of $\E$ bounded
by the two hyperplanes through $x$ and $y$, normal to $x-y$.

We choose $R>\Vert x-y\Vert/2$ small enough to ensure that
\[
K^{\prime}\eqd\conv(K\cup(\Bc(o,R)\setminus\Delta))
\]
satisfies $d_{PH}(K_{0},K_{1})<\varepsilon$. It remains to prove that a whole
neighbourhood of the poles $p^{+}\overset{\mathrm{def}}{=}o+\frac
{R}{\left\Vert o-x\right\Vert }\left(  x-o\right)  $ and $p^{-}\overset
{\mathrm{def}}{=}2o-p$ is included in $\IS(o,R)$. Let $B^{\pm}$ be the
connected component of $\Bc(o,R)\setminus\Delta$ that contains $p^{\pm}$.
Assume that there exists $p_{n}\in\IS(o,R)$, tending to $p$, and interior to
some line-segment $a_{n}b_{n}$ with $a_{n}\in B^{+}$ and $b_{n}\in\conv\left(
K\cup B^{-}\right)  $. Passing if necessary to a subsequence, we may assume
that $b_{n}$ converges to $b\in\conv\left(  K\cup B^{-}\right)  $. The
hyperplane $H_{n}$ through $p_{n}$ and normal to $\left(  x-y\right)  $
separates $a_{n}$ and $b_{n}$, and the connected component of $B^{+}\setminus
H_{n}$ containing $p$ tends to $\left\{  p\right\}  $, whence $a_{n}%
\rightarrow p$. Since $\left\Vert p_{n}-a_{n}\right\Vert \rightarrow0$,
$\measuredangle b_{n}a_{n}o\rightarrow\pi/2$. It follows that $b$ should
belong to the hyperplane through $p$ normal to $x-y$, and we get a contradiction.

Of course, the same proof holds for $p^{-}$.
\end{proof}

\begin{lem}
\label{lem:make-max-strict}Let $K$ be a convex body in $\E$ and $b_{1}$,
\ldots, $b_{n}\in\B(K)$ be $n$ double normals. Assume that each foot of
$b_{i}$ ($i\in\mathbb{N}_{n}$) admits a neighbourhood in $\bd K$ which does
not contain any line-segment. Then there exists a sequence $K_{p}\in\bds$
tending to $K$ when $p$ tends to $\infty$, such that $b_{1}$, \ldots$b_{n}$
belong to $\MS(K_{p})$ for any $p$.
\end{lem}

\begin{proof}
Let $u_{i}$, $v_{i}$ be the feet of $b_{i}$ ($i=1,\ldots,n$), and consider the
full half cone of revolution $C_{i,p}^{+}$ (respectively $C_{i,p}^{-}$) with
vertex $u_{i}$ (respectively $v_{i}$), axis $\overline{u_{i}v_{i}}$, angle
$\frac{\pi}{2}-\frac{1}{p}$ between the axis and the generatrices, and
containing $v_{i}$ (respectively $u_{i}$). Since $K$ is locally strictly
convex near $u_{i}$ and $v_{i}$, the intersection $K_{p}$ of $K$ and all these
cones tends to $K$ when $p$ tends to $\infty$ and clearly $b_{i}\in\MS(K_{p})$.
\end{proof}

\begin{thm}
\label{thm:UPD-bin}For most $K\in\bds$, we have $\dim_{P}\BF\left(  K\right)
=d$.
\end{thm}

\begin{proof}
Let $\mathcal{U}$ be a countable base of open sets of $\E$. For $N\geq1$ and
$V\in\mathcal{U}$, we define
\[
\Omega_{V,N}=\set(:K\in\bds|%
\begin{array}
[c]{l}%
\BF\left(  K\right)  \cap V=\emptyset\\
\text{or}\\
\exists\delta\in\left]  0,\frac{1}{N}\right[  \,\text{s.t. }\frac{\ln
P_{\delta}(\BF\left(  K\right)  \cap V)}{-\ln\delta}>d-\frac{1}{N}%
\end{array}
:)\text{.}%
\]
If for a given $V\in\mathcal{U}$, $K$ lies in the intersection of all these
$\Omega_{V,N}$, it satisfies $\overline{\dim}_{P}\BF\left(  K\right)  \cap
V=d$ whenever $\BF\left(  K\right)  \cap V$ is non-empty, and it follows that
$\dim_{P}\BF\left(  K\right)  =d$ by Lemma \ref{lem:dimP-Homo}. Thus we just
have to check the density in $\bds$ of $\inn\Omega_{V,N}$.

Let $V\in\mathcal{U}$, $N\geq1$, $K_{0}\in\bds$ and $\varepsilon>0$; we look
for some $K_{3}\in\mathring{\Omega}_{N}$ such that $d_{PH}(K_{0}%
,K_{3})<\varepsilon$. If $K_{0}$ belongs to $\inn\Omega_{V,N}$ then the proof
is over, otherwise there exits $K_{1}\in\bds$ such that $d_{PH}(K_{0}%
,K_{1})<\varepsilon$ and $\BF\left(  K_{1}\right)  \cap V$ contains at least
one element $\left(  x_{1},y_{1}\right)  $.

By Lemma \ref{lem:spher-cap}, there exist $K_{2}\in\bds$, $R>0$, $o\in K_{2}$
such that $d_{PH}\left(  K_{0},K_{2}\right)  <\varepsilon$ and $\bd K_{2}$
contains two open subsets $U^{\pm}$ of $S\overset{\mathrm{def}}{=}\IS(o,R)$,
image one to the other by the symmetry $\sigma:p\mapsto2o-p$, and such that
$U^{+}\subset V$. Choose $x\in U^{+}$ and let $r>0$ be small enough to ensure
that $C^{+}\overset{\mathrm{def}}{=}\Bc(x,r)\cap S\subset U^{+}$.

Since $\dim C^{+}=d$, we can choose $0<\delta<1/N$ such that
\[
{\frac{\ln P_{\delta}(C^{+})}{\ln2-\ln\delta}>d-\frac{1}{N}}%
\]
and a $\delta$-set $F\subset C^{+}$ with cardinality $P_{\delta}(C^{+})$.
Clearly $(x,\sigma\left(  x\right)  )\in\B(K_{1})$ for $x\in F$. By Lemma
\ref{lem:make-max-strict}, there exists $K_{3}\in\bds$ such that
$d_{PH}\left(  K_{0},K_{3}\right)  <\varepsilon$ and $(x,\sigma\left(
x\right)  )\in\MS(K_{3})$ for any $x\in F$.

By virtue of Lemma \ref{lem:stab}, there is a neighbourhood $V$ of $K_{3}$ in
$\bds$ such that for any $K\in V$, and any $x\in F$, there exists a double
normal $\left(  \tilde{x},\tilde{y}\right)  \in\B\left(  K\right)  $ verifying
$\widetilde{x}\in\Bo(x,\delta/4)\cap V$. From this we get that $P_{\delta
/2}(\BF\left(  K\right)  \cap V)\geq P_{\delta}(C^{+})$ and thus $K\in
\Omega_{V,N}$. Hence $K_{3}\in\inn\Omega_{N}$ and the proof is complete.
\end{proof}

\begin{rem}
\label{Rmk:diam} The reader may ask why this theorem is stated for $\BF$
instead of $\B$. As a matter of fact, obviously,
\[
\dim_{P}\BF\left(  K\right)  \leq\dim_{P}\B\left(  K\right)  \leq\dim
_{P}\Diam\left(  K\right)  \text{.}%
\]
Since this set is canonically one to one mapped (for a $\mathcal{C}^{1}$
strictly convex body $K$) on the unit sphere of $\E$, one may think that
$\overline{\dim}_{P}\Diam\left(  K\right)  =d$, in which case Theorem
\ref{thm:UPD-bin} would hold for $\B$ as well. However, this bijection is not
(known to be) regular enough to get any conclusion on the dimension of
$\Diam(K)$.

For a smooth strictly convex body, there is a diametral map $\Delta_{K}:\bd
K\rightarrow\bd K$ which associates to a point $x$ the only point $x^{\prime}$
such that $\left(  x,x^{\prime}\right)  \in\Diam\left(  K\right)  $. Hence
$\Diam\left(  K\right)  \subset\bd K\times\bd K$ is the graph of this map.
However, this map is not necessarily Lipschitz continuous, or regular enough
to carry any dimensional information. Indeed, K. Adiprasito and T. Zamfirescu
proved that it behaves rather badly in the typical case, for it maps a set of
full measure on a set of measure zero \cite{AZ1}.

Nevertheless, if $d=1$, an elementary argument of monotony of $\Delta_{K}$
shows that $\dim\Diam\left(  K\right)  =1$ for any reasonable notion of
dimension. See Lemma \ref{lem:dim-diam}.
\end{rem}

\section{Critical values\label{sec:crit-val}}

This section focuses on the set of lengths of double normals. As seen earlier,
double normals can be seen as critical points of the length function, so their
lengths are critical values.

\begin{lem}
\label{lem:supstd-dense}There is an open and dense set $U\subset\E^{n}$ such
that for any four pairwise disjoint non-empty sets of indices $I$, $J$,
$I^{\prime}$, $J^{\prime}\subset\mathbb{N}_{n}$ of cardinality at most $d$,
the distance between $\left\langle x_{I}\right\rangle $ and $\left\langle
x_{J}\right\rangle $ and the distance between $\left\langle x_{I^{\prime}%
}\right\rangle $ and $\left\langle x_{J^{\prime}}\right\rangle $ are either
distinct or both equal to $0$.
\end{lem}

\begin{proof}
There is an open and dense set $V_{0}\subset\E^{n}$ such that for any
non-empty set of indices $I\subset\mathbb{N}_{n}$, $\dim\left\langle
x_{I}\right\rangle =\#I-1\text{. If }\#I+\#J>d+3$ then $d\left(  \left\langle
x_{I}\right\rangle ,\left\langle x_{J}\right\rangle \right)  =0$ for any $x\in
V_{0}$. So, from now on, we assume implicitly that $\#I+\#J\leq d+3$ and
$\#I^{\prime}+\#J^{\prime}\leq d+3$. Now, by Lemma \ref{lem:std-dense}, there
is an open and dense set $V_{1}\subset V_{0}$ such that for any disjoint sets
of indices $I$, $J\subset\mathbb{N}_{n}$, $\overrightarrow{x_{I}}%
\cap\overrightarrow{x_{J}}=\left\{  0\right\}  $. Moreover, there exists a
real valued rational function $P_{IJ}$ on $\E^{n}$ whose restriction to
$V_{1}$ satisfies $d\left(  \left\langle x_{I}\right\rangle ,\left\langle
x_{J}\right\rangle \right)  ^{2}=P_{IJ}\left(  x\right)  $. We have to prove
that, given four pairwise disjoint sets of indices $I$, $J$, $I^{\prime}$,
$J^{\prime}$, the open set $U_{II^{\prime}JJ^{\prime}}=\set(:x\in V_{1}%
|P_{IJ}\left(  x\right)  \neq P_{I^{\prime}J^{\prime}}\left(  x\right)  :)$ is
dense in $V_{1}$. Since it is defined by polynomials inequations, it is
sufficient to prove that it is not empty. This latter fact is obvious because
the sets of indices are disjoint.
\end{proof}

\begin{thm}
\label{thm:inj}For most $K\in\bds$, $\widetilde{\lgt }_{K}:\Bno(K)\to\R$ is injective.
\end{thm}

\begin{proof}
A set $K$ has a non-injective function $\widetilde{\lgt}_{K}$ if and only if
there exists an integer $n$ and two non-oriented double normals $b_{1}$,
$b_{2}$ such that $d_{PH}\left(  b_{1},b_{2}\right)  \geq\frac{1}{n}$ and
$\lgt\left(  b_{1}\right)  =\lgt\left(  b_{2}\right)  $. For fixed $n$, the
set $\mathcal{A}_{n}$ of such bodies is obviously closed in $\bds^{d}$. Since
a double normal realizes the distance between the affine spaces spanned by two
disjoint faces (disjoint, because they lie in two parallel hyperplanes), by
Lemma \ref{lem:supstd-dense}, there is a dense set of polytopes that does not
intersect $\mathcal{A}_{n}$ and the proof is finished.
\end{proof}

\begin{cor}
\label{cor:MMS}For most $K\in\bds$, $\M(K)=\MS(K)$.
\end{cor}

\begin{cor}
\label{cor:LPDvalue}For most $K\in\bds$, $\sp(K)$ is homeomorphic to the
Cantor set and has lower box-counting dimension $0$.
\end{cor}

\begin{proof}
It's easy to see from Lemmas \ref{lem:tot-disc} and \ref{lem:perfect} that
$\Bno\left(  K\right)  $ is a Cantor set. Since, by Theorem \ref{thm:inj},
$\widetilde{\ell}:\Bno\left(  K\right)  \rightarrow\R$ is injective,
$\sp(K)=\widetilde{\ell}(\Bno\left(  K\right)  )=\lgt(\B\left(  K\right)  )$
is also a Cantor set. Moreover, $\lgt$ is Lipschitz continuous, whence, by
Theorem \ref{T5},
\[
\underline{\dim}_{P}\sp(K)\leq\underline{\dim}_{P}\B\left(  K\right)
=0\text{.}%
\]

\end{proof}

Concerning the upper dimension, we get the following result.

\begin{thm}
\label{thm:UPD-values} For most $K\in\bds$,

\begin{itemize}
\item if $d=1$, $\dim_{P}\sp\left(  K\right)  =\frac{1}{2}$,

\item if $d=2$, $\dim_{P}\sp\left(  K\right)  \geq\frac{3}{4}$,

\item if $d\geq3$, $\dim_{P}\sp\left(  K\right)  =1$.
\end{itemize}
\end{thm}

\begin{rem}
\label{Rmk:34} We conjecture that, in the case $d=2$, $\dim_{P}\sp\left(
K\right)  $ cannot exceed $3/4$ for any $K\in\mathcal{K}$. Obviously the
conjecture implies the equality in Theorem \ref{thm:UPD-values}.
\end{rem}

The rest of the section is devoted to the proof and will be divided in several
lemmas; the final compilation is postponed to the end of the section.

\begin{lem}
\label{lem:dim-diam}If $d=1$ and $K\in\bds$ is $\mathcal{C}^{1}$ then
$\overline{\dim}_{B}\Diam\left(  K\right)  =1$.
\end{lem}

\begin{proof}
Let $\Delta_{K}:\bd K\rightarrow\bd K$ be the function which associates to
each point $x$ the other extremity of the affine diameter starting at $x$.
Thus $\Diam\left(  K\right)  $ is the graph of $\Delta_{K}$.

It is easy to see that two distinct affine diameters of $K$ always intersect
inside $K$.

Thus $\Delta_{K}$ is locally monotone, in the following sense: for any
homeomorphisms $\phi:\left[  0,1\right]  \rightarrow U\subset\bd K$,
$\psi:\left[  0,1\right]  \rightarrow V\subset\bd K$ such that $x\in U$ and
$\Delta_{K}\left(  x\right)  \in V$, $\psi^{-1}\circ\Delta_{K}\circ\phi$ is
monotone. It follows that the dimension of the graph of $\Delta_{K}$, cannot
exceed the dimension of $\bd K$.
\end{proof}

Let $\mathcal{V}$ be a basis of open sets of $\mathbb{R}$. For $V\in
\mathcal{V}$ and $N$ a positive integer, define%

\begin{align*}
U_{V,N}^{\kappa}  &  \eqd\set(:K\in\bds|\exists\delta\in\left]  0,\frac{1}%
{N}\right[  \text{s.t. }\frac{\ln P_{\delta}(\lgt(\MS(K))\cap V)}{-\ln
(\delta/2)}>\kappa-\frac{1}{N}:)\text{,}\\
W_{V}  &  \eqd\set(:K\in\bds|\sp(K)\cap V=\emptyset:)\text{.}%
\end{align*}

\begin{lem}
\label{lem:densU-d1}For $d=1$, for all $V\in\mathcal{V}$ and $N>0$,
$U_{V,N}^{1/2}$ is dense in $\bds\setminus W_{V}$.
\end{lem}

\begin{proof}
Fix $K_{0}\in\bds\setminus W_{V}$, $\left(  x_{0},y_{0}\right)  \in\B\left(
K\right)  $ such that $\left\Vert x_{0}-y_{0}\right\Vert \in V$, and
$\varepsilon>0$. By Lemma \ref{lem:spher-cap}, there exits $K_{1}\in\bds$ such
that $d_{PH}\left(  K_{0},K_{1}\right)  <\varepsilon$ and $\bd K_{1}$ contains
two circle arcs $C^{\pm}$ sharing the same centre $o$, symmetrical to each
other with respect to $o$, and such that the line $x_{0}y_{0}$ intersects
$\bd
K_{1}$ in two points $x\in C^{+}$ and $y\in C^{-}$. We may also assume that
$2R\overset{\mathrm{def}}{=}\left\Vert x-y\right\Vert \in V$. By considering
even smaller arcs, one can assume without loss of generality that $x$ is the
midpoint of $C^{+}$; let $a$, $b$ be its extremities. Put
$\Theta=\measuredangle xoa$. Making if necessary $C^{+}$ even smaller, we may
also assume without loss of generality that
\[
\cos\theta\leq1-\theta^{2}/3
\]
for any $\theta\in\left[  0,\Theta\right]  $. The lines tangent to $C^{+}$ at
$a$ and $b$ intersect at some point $c$ collinear with $o$ and $x$. Note that
the union $K_{2}$ of the triangle $abc$ and $K_{1}$ is convex. Making if
necessary $C^{+}$ even smaller, we may assume without loss of generality that
$d_{PH}\left(  K_{0},K_{2}\right)  <\varepsilon$. Let $u\left(  \theta\right)
$ be a unit vector such that $\measuredangle\left(  u\left(  \theta\right)
,a-o\right)  =\theta$ and $Ru\left(  \theta\right)  \in C^{+}$. Now choose a
positive integer $n$ and define for $i=0,\ldots,n$%
\begin{align*}
\delta &  =R\Theta^{2}/4n^{2}\text{,}\\
R_{i}  &  =R+i\delta\text{,}\\
v_{i}  &  =o+R_{i}u\left(  i\Theta/n\right)  \text{.}%
\end{align*}
Note that, for $i>0$,
\[
R_{i}\cos\frac{i\Theta}{n}<R\left(  1-\frac{\Theta^{2}}{12n^{2}}\right)
<R\text{,}%
\]
whence all the $v_{i}$ belong to the triangle $abc$. Let $K_{3}$ be the convex
hull of $K_{2}$, the points $v_{i}$ and their symmetrical points
$v_{i}^{\prime}$ with respect to $o$. Since $K_{1}\subset K_{3}\subset K_{2}$
we have $d_{PH}\left(  K_{0},K_{3}\right)  <\varepsilon$.

We claim that any triangle $ov_{i}v_{j}$ with $1\leq j<i\leq n$ is acute.
Since $R_{i}>R_{j}$, it is clear that $\measuredangle ov_{i}v_{j}<\pi/2$.
Moreover $\measuredangle ov_{j}v_{i}$ is acute if and only if $\frac{R_{j}%
}{R_{i}}>\cos\frac{\left(  j-i\right)  \Theta}{n}$. Now%
\[
\frac{R_{j}}{R_{i}}=\frac{R_{j}-\left(  j-i\right)  \delta}{R_{i}}%
=1-\frac{\left(  j-i\right)  \Theta^{2}}{4n^{2}}\frac{R}{R_{i}}\geq
1-\frac{\Theta^{2}\left(  j-i\right)  }{4n^{2}}\text{.}%
\]
On the other hand
\[
\cos\frac{\left(  j-i\right)  \Theta}{n}\leq1-\frac{\left(  j-i\right)
^{2}\Theta^{2}}{3n^{2}}\leq1-\frac{\left(  j-i\right)  \Theta^{2}}{3n^{2}%
}\text{,}%
\]
and the claim is proven. Moreover, $\measuredangle v_{n}ob<\pi/2$ because
$r_{n}>R$. It follows that $v_{i}\in\bd K_{3}$ and $\left(  v_{i}%
,v_{i}^{\prime}\right)  \in\MS(K_{3})$, $1\leq i\leq n$. Obviously the set
$\{\lgt(v_{i})|i\in\mathbb{N}_{n}\}$ is a $\delta$-set; for $n$ large enough,
it is included in $V$. It follows that $P_{\delta}(\lgt(\MS(K_{3}))\cap V)\geq
n$. Since $\lim\frac{\ln n}{-\ln\delta}=\frac{1}{2}$, for $n$ large enough,
$K_{3}\in U_{V,N}^{1/2}$.
\end{proof}

\begin{lem}
\label{lem:densU-d3}For $d\geq3$, for all $V\in\mathcal{V}$, $U_{V,N}^{1}$ is
dense in $\bds\setminus W_{V}$.
\end{lem}

\begin{proof}
Choose $K_{0}\in\bds\setminus W_{V}$, $b\in\B\left(  K_{0}\right)  $ such that
$\lgt_{K_{0}}\left(  b\right)  \in V$, and $\varepsilon>0$. We have to prove
that there exists $K\in U_{V,N}^{1}$ such that $d_{H}(K_{0},K)<\varepsilon$.

By \textquotedblleft combination\textquotedblright\ of $K_{0}$ and the convex
bodies $K^{\star}$ provided by Theorem \ref{qthm:Kuiper}, one can find $K_{1}$
such that $d(K_{0},K_{1})<\varepsilon$ and $\sp\left(  K_{1}\right)  \cap V$
contains an interval $[a,a+2\Delta]$. Here the convex bodies are combined
using the same construction as in the proof of Lemma \ref{lem:spher-cap} at
the neighbourhood of $b$, replacing the sphere by a rescaled and displaced
copy of $K^{\star}$.

Put $\delta_{0}\eqd\frac{2\Delta}{M}$, where $M$ is chosen large enough to
ensure that $\delta_{0}<\frac{1}{N}$ and
\[
\frac{\ln\Delta}{\ln\Delta-\ln M}<\frac{1}{N}\text{.}%
\]
Let $b_{i}\in\B(K_{1})$ ($i=0,\ldots,M$) be a double normal of length
$a+i\delta_{0}$. By Lemma \ref{lem:make-max-strict}, one can find $K_{2}%
\in\bds$ such that $d_{H}(K_{0},K_{2})<\varepsilon$ and $b_{i}\in\MS(K_{2})$.
Now,
\[
P_{\delta_{0}}(\lgt(\MS(K_{2}))\cap V)\geq M=\frac{\Delta}{\delta_{0}%
/2}\text{,}%
\]
whence
\[
\frac{\ln(P_{\delta_{0}}(\lgt(\MS(K_{2}))\cap V))}{-\ln(\delta_{0}/2)}%
\geq1-\frac{\ln\Delta}{\ln\Delta-\ln M}>1-\frac{1}{N}%
\]
by the choice of $M$. Hence $K_{2}$ belongs to $U_{V,N}^{1}$ and the proof is complete.
\end{proof}

The next technical lemma is needed for the case $d=2$.

\begin{lem}
\label{lem:dimsp-d2-angles}Consider the classical parametrization of the unit
sphere
\[
\phi:(\lambda,\theta)\mapsto(\cos\lambda\cos\theta,\cos\lambda\sin\theta
,\sin\lambda)\text{,}%
\]
choose $R>0$, $A>0$, $T\in\left(  0,\frac{\pi}{4}\right)  $ and define for any
natural integer $m$ and any $\left(  i,j\right)  \in\mathbb{N}_{m^{2}}%
^{0}\times\mathbb{N}_{m}^{0}$%
\begin{align*}
\delta &  =\frac{RT^{2}}{16m^{4}}\text{,}\\
r_{ij}  &  =R+A-\left(  jm^{2}+i\right)  \delta\text{,}\\
v_{ij}  &  =r_{ij}\phi\left(  \frac{iT}{m^{2}},\frac{jT}{m}\right)
\hfill\text{.}%
\end{align*}
Then, for $m$ large enough, for any $\left(  i,j\right)  \neq\left(
i^{\prime},j^{\prime}\right)  \in\mathbb{N}_{m^{2}}^{0}\times\mathbb{N}%
_{m}^{0}$,
\[
\left\langle v_{i^{\prime}j^{\prime}},v_{i^{\prime}j^{\prime}}-v_{ij}%
\right\rangle >0\text{.}%
\]

\end{lem}

\begin{proof}
Put%
\[
D=\left\langle v_{i^{\prime}j^{\prime}},v_{i^{\prime}j^{\prime}}%
-v_{ij}\right\rangle \text{.}%
\]
Since for $i+m^{2}j\leq i^{\prime}+m^{2}j^{\prime}$, $\left\Vert v_{i^{\prime
}j^{\prime}}\right\Vert \leq\left\Vert v_{i,j}\right\Vert $, it is sufficient
to check the sign of $D$ for $i^{\prime}+m^{2}j^{\prime}>i+m^{2}j$. A
straightforward computation shows that%
\[
D=r^{\prime}(r^{\prime}-Pr)\text{,}%
\]
with%

\begin{align*}
r  &  =r_{ij}\text{,}\\
r^{\prime}  &  =r_{i^{\prime}j^{\prime}}\text{,}\\
P  &  =\cos\frac{Ti}{m^{2}}\cos\frac{Ti^{\prime}}{m^{2}}\cos\frac{T\left(
j^{\prime}-j\right)  }{m}+\sin\frac{Ti}{m^{2}}\sin\frac{Ti^{\prime}}{m^{2}%
}\text{.}%
\end{align*}
Thus $D>0$ if and only if%
\[
P<\frac{r^{\prime}}{r}\text{.}%
\]
We claim that these inequalities hold for $m$ large enough, for any $\left(
i,j\right)  \neq\left(  i^{\prime},j^{\prime}\right)  \in\mathbb{N}_{m^{2}%
}^{0}\times\mathbb{N}_{m}^{0}$. Assume, on the contrary, that there exist
sequences $m_{p}$, $i_{p}$, $j_{p}$, $i_{p}^{\prime}$ and $j_{p}^{\prime}$
such that $m_{p}\rightarrow\infty$, $i/m_{p}^{2}$, $i_{p}^{\prime}/m_{p}^{2}$,
$j_{p}/m_{p}$, $j_{p}^{\prime}/m_{p}\in\left[  0,1\right]  $ and the
corresponding value of $D$ is non-positive. Extracting if necessary
subsequences, one may assume without loss of generality that the four ratios
are converging in $\left[  0,1\right]  $; denote by $\alpha$, $\alpha^{\prime
}$, $\beta$ and $\beta^{\prime}$ the respective limits of $Ti_{p}/m_{p}^{2}$,
$Ti_{p}^{\prime}/m_{p}^{2}$, $Tj_{p}/m_{p}$ and $Tj_{p}^{\prime}/m_{p}$. Then
$P$ converges to
\[
\cos\alpha\cos\alpha^{\prime}\cos\left(  \beta^{\prime}-\beta\right)
+\sin\alpha\sin\alpha^{\prime}\leq1\text{,}%
\]
with equality if and only if $\alpha^{\prime}=\alpha$ and $\beta^{\prime
}=\beta$. On the other hand, $r^{\prime}/r$ tends to $1$. It follows that, if
$\alpha\neq\alpha^{\prime}$ or $\beta\neq\beta^{\prime}$, a contradiction is
found. From now on, we assume $\alpha^{\prime}=\alpha$ and $\beta^{\prime
}=\beta$. We now discuss two cases.

\textbf{Case 1.} There exists an arbitrarily large $p$ such that $j_{p}%
=j_{p}^{\prime}$. By extracting suitable subsequences, we may assume without
loss of generality that $j_{p}=j_{p}^{\prime}$ (and so $i_{p}^{\prime}>i_{p}$)
for all $p$. Then, since $\frac{T\left(  i_{p}^{\prime}-i_{p}\right)  }%
{m_{p}^{2}}\rightarrow\alpha^{\prime}-\alpha=0$, for $m$ large enough,
\[
P=\cos\frac{T\left(  i_{p}^{\prime}-i_{p}\right)  }{m_{p}^{2}}<1-\frac
{T^{2}\left(  i_{p}^{\prime}-i_{p}\right)  ^{2}}{4m_{p}^{4}}\leq1-\frac
{T}{4m_{p}^{4}}^{2}\left(  i_{p}^{\prime}-i_{p}\right)  \text{.}%
\]
On the other hand,
\[
\frac{r^{\prime}}{r}=\frac{r-\left(  i_{p}^{\prime}-i_{p}\right)  \delta}%
{r}>1-\frac{T^{2}}{16m_{p}^{4}}\left(  i_{p}^{\prime}-i_{p}\right)
\]
and we get a contradiction.

\textbf{Case 2. }For $p$ large enough, $\hat{\jmath}_{p}\overset{\mathrm{def}%
}{=}j_{p}^{\prime}-j_{p}>0$. By extracting suitable subsequences, we may
assume without loss of generality that this inequality holds for all $p$.
Define $\alpha_{p}$, $\hat{\alpha}_{p}$ and $\hat{\beta}_{p}$ by $i_{p}%
=\alpha_{p}m_{p}^{2}/T$, $i_{p}^{\prime}=\left(  \alpha_{p}+\hat{\alpha}%
_{p}\right)  m_{p}^{2}/T$ and $\hat{\jmath}_{p}=m_{p}\hat{\beta}_{p}$; then
$\lim_{p\rightarrow\infty}\hat{\alpha}_{p}=\lim_{p\rightarrow\infty}\hat
{\beta}_{p}=0$ and by straightforward computations%
\begin{align*}
P  &  =\sin2\alpha_{p}\sin\hat{\alpha}_{p}\sin^{2}\frac{\hat{\beta}_{p}}%
{2}+Q\cos\hat{\alpha}_{p}\leq\frac{\hat{\beta}_{p}^{2}}{4}\left\vert
\hat{\alpha}_{p}\right\vert +Q\text{,}\\
Q  &  =\frac{1}{2}\left(  \left(  \cos2\alpha_{p}+1\right)  \cos\hat{\beta
}_{p}-\cos2\alpha_{p}+1\right)  \text{.}%
\end{align*}
Since $\hat{\beta}_{p}\rightarrow0$, for $p$ large enough $\cos\hat{\beta}%
_{p}<1-\hat{\beta}_{p}^{2}/3$ whence%
\[
Q<1-\frac{\hat{\beta}_{p}^{2}}{6}\text{.}%
\]
For $p$ large enough, $\left\vert \hat{\alpha}_{p}\right\vert <\frac{2}{21}$,
whence%
\[
P<1-\frac{\hat{\beta}_{p}^{2}}{7}=1-\frac{\hat{\jmath}_{p}^{2}T^{2}}%
{7m_{p}^{2}}\leq1-\frac{\hat{\jmath}_{p}T^{2}}{7m_{p}^{2}}\text{.}%
\]

On the other hand%
\begin{align*}
\frac{r^{\prime}}{r}  &  =\frac{r-\left(  \hat{\jmath}_{p}m^{2}+i_{p}^{\prime
}-i_{p}\right)  \delta}{r}\\
&  \geq1-\left(  \hat{\jmath}_{p}+1\right)  \frac{T^{2}}{16m_{p}^{2}}%
\geq1-\frac{\hat{\jmath}_{p}T^{2}}{8m_{p}^{2}}\text{,}%
\end{align*}
and we get another contradiction. This completes the proof.
\end{proof}

\begin{lem}
\label{lem:densU-d2}For $d=2$, for all $V\in\mathcal{V}$, $U_{N}^{3/4}$ is
dense in $\bds\setminus W_{V}$.
\end{lem}

\begin{proof}
Choose $K_{0}\in\bds$ $\left(  x_{0},y_{0}\right)  \in\B\left(  K_{0}\right)
$ such that $\left\Vert x_{0}-y_{0}\right\Vert \in V$, and $\varepsilon>0$; we
have to prove that there exists $K\in U_{V,N}^{3/4}$ such that $d_{H}%
(K_{0},K)<\varepsilon$. By Lemma \ref{lem:spher-cap}, one can find a convex
body $K_{1}$ whose distance from $K_{0}$ is less than $\varepsilon$ and whose
boundary contains two spherical caps, symmetrical to each other with respect
to some point $o$. Let $R$ be the radius of this sphere; we may assume that
$2R\in V$. One can also assume, without loss of generality, that $o=(0,0,0)$
and that those caps are centered at equatorial points $\pm e=(\pm R,0,0)$.
Denote by $C$ the cap centered at $e$, and, for $A>0$, by $\hat{C}$ the convex
hull of $C\cup\left\{  \left(  R+2A,0,0\right)  \right\}  $. For $A$
sufficiency small $K_{2}=K_{1}\cup\hat{C}_{A}\cup\left(  -\hat{C}_{A}\right)
$ is convex and $d_{PH}\left(  K_{0},K_{2}\right)  <\varepsilon$. Let $\phi$
be a classical parametrization of the unit sphere:%
\[
\phi(\lambda,\theta)=(\cos\lambda\cos\theta,\cos\lambda\sin\theta,\sin
\lambda)\text{.}%
\]
Let $T>0$ be small enough to ensure that $\left(  R+A\right)  \phi\left(
\left[  0,T\right]  \times\left[  0,T\right]  \right)  $ is included in the
interior of $\hat{C}\setminus K_{1}$. For any positive integer $m$, and any
$(i,j)\in\mathbb{N}_{m^{2}}^{0}\times\mathbb{N}_{m}^{0}$, define
\begin{align*}
\delta &  =\frac{RT^{2}}{16m^{4}}\text{,}\\
r_{ij}  &  =R+A-\left(  jm^{2}+i\right)  \delta\text{,}\\
v_{ij}  &  =r_{ij}\phi\left(  \frac{iT}{m^{2}},\frac{jT}{m}\right)  \text{.}%
\end{align*}
For $m$ large enough, all the $v_{ij}$ lie in $\hat{C}$ and $V\overset
{\mathrm{def}}{=}\conv\left\{  v_{ij}\right\}  _{\substack{i\in\mathbb{N}%
_{m^{2}}-1\\j\in\mathbb{N}_{m}-1}}$ does not intersect $K_{1}$. Let $K_{3}$ be
the convex hull of $V$ and $K_{1}$. Since $K_{1}\subset K_{3}\subset K_{2}$,
$d_{PH}\left(  K_{0},K_{3}\right)  <\varepsilon$. By Lemma
\ref{lem:dimsp-d2-angles}, for $m$ large enough%
\[
\left\langle v_{ij},v_{ij}-v_{i^{\prime}j^{\prime}}\right\rangle >0\text{.}%
\]
Moreover, for $c\in C$%
\[
\left\langle v_{ij},v_{ij}-c\right\rangle >0
\]
because $\left\Vert v_{ij}\right\Vert =r_{ij}>R=\left\Vert c\right\Vert $.
Thus for any point $p\neq v_{ij}$ in%
\[
G\overset{\mathrm{def}}{=}\hat{C}\cap K_{3}=\conv\left(  C\cup V\right)
\text{,}%
\]
the angle $\measuredangle ov_{ij}p$ is less than $\pi/2$. It follows that
$v_{ij}\in\bd K_{3}$ and that $\left(  v_{ij},-v_{ij}\right)  $ are maximizing
chords of $K_{3}$.

For $m$ large enough, all the lengths of those chords belong to $V$, whence
\[
P_{\delta}(\lgt(\MS(K)\cap V)\geq m^{3}%
\]
and
\[
\frac{\ln P_{\delta}(\lgt(\MS(K)\cap V)}{-\ln(\delta/2)}=\frac{3\ln m}{4\ln
m-\ln\frac{RT^{2}}{32}}\underset{m\rightarrow\infty}{\rightarrow}\frac{3}%
{4}\text{,}%
\]
whence $K_{3}$ belongs to $U_{V,N}^{3/4}$ if $m$ is large enough. This ends
the proof.
\end{proof}

\begin{proof}
[Proof of Theorem \ref{thm:UPD-values}]By Lemmas \ref{lem:dim-diam} and
\ref{lem:Holder}, $\overline{\dim}_{B}\left(  K\right)  \leq1/2$ for $d=1$.
Clearly this dimension is upper bounded by $1$ in any case. So we just have to
prove that $\dim_{P}\left(  \sp(K)\right)  \geq d^{\ast}\overset{\mathrm{def}%
}{=}\min\left(  1,\frac{1+d}{4}\right)  $.

For $N\geq1$ and $V\in\mathcal{V}$ define
\[
\Omega_{V,N}\eqd\set(:K\in\bds|%
\begin{array}
[c]{l}%
\exists\delta\in\left]  0,\frac{1}{N}\right[  \text{s.t.}\frac{\ln P_{\delta
}(\sp(K)\cap V)}{-\ln\delta}>d^{\ast}-\frac{1}{N}\\
\text{or}\\
\sp(K)\cap V=\emptyset
\end{array}
:)\text{.}%
\]
If, for a fixed $V\in\mathcal{V}$, $K$ lies in infinitely many $\Omega_{V,N}$
then $\overline{\dim}_{B}\left(  K\cap V\right)  \geq d^{\ast}$ whenever
$\sp(K)\cap V$ is not empty. It follows by Lemma \ref{lem:dimP-Homo} that
$\dim_{P}\BF\left(  K\right)  =d^{\ast}$, for any $K$ liying in the
intersection of all $\Omega_{V,N}$, $V\in\mathcal{V}$, $N>1$. Thus we just
have to check the density of the interior of $\Omega_{V,N}$ in $\bds$.

If $K_{0}\in U_{V,N}^{d^{\ast}}$, then there exit $\delta<\frac{1}{N}$ and
$M\overset{\mathrm{def}}{=}1+\left\lceil \left(  \frac{\delta}{2}\right)
^{-d^{\ast}+1/N}\right\rceil $ double normals $b_{1}$, \ldots, $b_{M}$ whose
lengths form a $\delta$-set. Hence, by Lemma \ref{lem:stab}, for $K$ close
enough to $K_{0}$, there exist $M$ double normals of $K$ whose lengths form a
$\delta/2$-set, thus $P_{\delta/2}\left(  K\right)  \geq M$ and $K\in
\Omega_{V,N}$. It follows that $U_{V,N}^{d^{\ast}}$ is included in the
interior of $\Omega_{V,N}$.

By Lemmas \ref{lem:densU-d1}, \ref{lem:densU-d2} and \ref{lem:densU-d3},
$U_{V,N}^{d^{\ast}}$ is dense in $\bds\setminus W_{V}$, whence, $U_{V,N}%
^{d^{\ast}}\cup\inn W_{V}$ is dense in $\bds$ and included in $\Omega_{V,N}$.
\end{proof}

\section{Critical points\label{sec:KoCP}}

As Gruber showed in \cite{Grub-B}, a typical convex body $K$ is not
$\mathcal{C}^{2}$. It follows that the usual classification of critical points
of $\lgt_{K}$ according to the Hessian does not work. However, one can
distinguish local maxima, local minima, and other critical points. Since the
curvature (and so the Hessian) is typically undefined, it is unclear weather
those other critical points look like saddles.

The first proposition is obvious and its proof is left to the reader.

\begin{prop}
For a strictly convex body $K\in\bds$, $\lgt _{K}$ has no local minimum.
\end{prop}

Local maxima are not very numerous either.

\begin{prop}
\label{prp:countable}For most convex bodies $K\in\bds$, the set $\M(K)$ is at
most countable.
\end{prop}

\begin{proof}
By Theorem \ref{thm:inj}, for most $K\in\bds$, $\widetilde{\lgt}%
_{K}:\Bno(K)\rightarrow\R$ is injective. Let $\mathcal{W}_{K}$ be a countable
base of open sets of $\C\left(  K\right)  \setminus\set(:(x,x)|x\in\bd K:)$
such that $\left(  x,y\right)  \in V\in\mathcal{W}_{K}$ implies $\left(
y,x\right)  \notin V$. Let $\mathcal{W}_{K}^{\prime}\subset\mathcal{W}_{K}$ be
the subset of those $V$ such that $\lgt_{K}|V$ admits a maximum, which is
necessarily unique by the injectivity of $\widetilde{\lgt}_{K}$. Then the map
$\mathcal{W}_{K}^{\prime}\rightarrow\M\left(  K\right)  $ mapping $V$ to this
maximum is surjective and the proof is complete.
\end{proof}

However, we have the following proposition.

\begin{prop}
\label{prp:dense}For most $K\in\bds$, $\M^{S}\left(  K\right)  $ is dense in
$\B(K)$.
\end{prop}

\begin{proof}
Let $\bds^{S}$ be the set of convex bodies $K$ such that $\MS(K)=\M(K)$. By
Corollary \ref{cor:MMS}, $\bds^{S}$ is residual in $\bds$, so by Lemma
\ref{lem:Z}, it is sufficient to prove the conclusion for most $K\in\bds^{S}$.
Let $\mathcal{U}^{2}$ be a countable basis of open sets of $\E^{2}$. For
$U\in\mathcal{U}^{2}$, define%
\begin{align*}
\Phi_{U}  &  \overset{\mathrm{def}}{=}\set(:K\in\bds^{S}|\B\left(  K\right)
\cap\overline{U}=\emptyset:)\\
\Psi_{U}  &  \overset{\mathrm{def}}{=}\set(:K\in\bds^{S}|\M\left(  K\right)
\cap U\neq\emptyset:)\text{.}%
\end{align*}

Those sets are open in $\bds^{S}$ by Lemma \ref{lem:USC} and Lemma
\ref{lem:stab} respectively. If $K$ belongs to the $G_{\delta}$-set
$\bigcap_{U\in\mathcal{U}^{2}}\left(  \Phi_{U}\cup\Psi_{U}\right)  $, then
$\M\left(  K\right)  $ is dense in $\B\left(  K\right)  $. Hence, it is
sufficient to prove that $\Psi_{U}\cup\Phi_{U}$ is dense in $\bds^{S}$. Choose
$K_{0}\in\bds^{S}$ and a neighbourhood $\mathcal{O}$ of $K_{0}$ in $\bds^{S}$.
We have to find $K_{3}\in\left(  \Phi_{U}\cup\Psi_{U}\right)  \cap\mathcal{O}%
$. First we chose a polytope $K_{1}\in\mathcal{O}$ (By corollary
\ref{cor:polyMMS}, polytopes all belong to $\bds^{S}$). If $K_{1}\in\Phi_{U}$
put $K_{3}=K_{1}$ and the proof is finished; otherwise there exists a double
normal of $K_{1}$ liying in $\overline{U}$. In this case, one can sightly
dilate and move $K_{1}$ in order to obtain another polytope $K_{2}%
\in\mathcal{O}$ admitting a double normal $\left(  x,y\right)  \in U$. For
$\eta>0$, define $x^{\prime}\overset{\mathrm{def}}{=}x+\eta\left(  x-y\right)
$, $y^{\prime}\overset{\mathrm{def}}{=}y+\eta\left(  y-x\right)  $ and
$K_{3}\overset{\mathrm{def}}{=}\conv\left(  K_{2}\cup\left\{  x^{\prime
},y^{\prime}\right\}  \right)  $. By Lemma \ref{lem:strict-max-criterion},
$\left(  x^{\prime},y^{\prime}\right)  \in\M\left(  K\right)  $. If $\eta$ is
small enough, then $K_{3}$ still belongs to $\mathcal{O}$ and $\left(
x^{\prime},y^{\prime}\right)  \in U$, whence $K_{3}\in\Psi_{U}\cap\mathcal{O}$.
\end{proof}

\begin{rem}
\label{rmk:index}In the case $d=1$, if $K$ is $\mathcal{C}^{2}$ the Hessian of
$\lgt_{K}$ at $b=(x,y)\in\B(K)$ is given by
\[
\left(
\begin{array}
[c]{cc}%
\frac{1}{w}-\gamma_{x} & \frac{1}{w}\\
\frac{1}{w} & \frac{1}{w}-\gamma_{y}%
\end{array}
\right)  \text{,}%
\]
where $\gamma_{u}$ is the curvature of $\bd K$ at $u=x,y$ and $w=\left\Vert
x-y\right\Vert $. Hence the Hessian degenerates when
\[
\frac{1}{\gamma_{x}}+\frac{1}{\gamma_{y}}=w\text{.}%
\]
So, the index of a double normal seen as a critical point appears to be
closely related to the curvature of $\bd K$ at its feet. This contributes to
the motivation for the following section. See also \cite{KoC}, \cite{Koz},
\cite{Koz-d}.
\end{rem}

\section{Curvature at feet of double normals\label{sec:curv-feet}}

This section brings some light on the curvature aspect of most convex
surfaces, at the endpoints of their double normals.

For distinct $x,y\in\E$, let $C_{xy}=\IS\left(  x,\left\Vert x-y\right\Vert
\right)  $ be the sphere of centre $x$ passing through $y$.


On a $\mathcal{C}^{1}$ convex body $K$, the unit sphere of the tangent space
to $\bd K$ at $x\in K$ is denoted by $\TUS_{x}K$. Consider a smooth, strictly
convex body $K$, a point $x$ on its boundary $\bd K$, and a tangent direction
$\tau\in\TUS_{x}K$. Take the 2-dimensional half-plane $H$ whose boundary line
$N$ is along the normal at $x$, such that $x+\tau\in H$. Then, for any point
$z\in H\cap\bd K$, there is exactly one circle with its centre on $N$ and
containing both $x$ and $z$. Let $r_{z}$ be the radius of this circle. Then
$\lcr\tau(x)=\liminf_{z\rightarrow x}r_{z}$ is called the \emph{lower
curvature radius} at $x$ in direction $\tau$. Analogously is defined the
\emph{upper curvature radius} $\ucr\tau(x)$. Also, $\lc\tau(x)=\ucr\tau
(x)^{-1}$ and $\uc\tau(x)=\lcr\tau(x)^{-1}$ are the \emph{lower} and
\emph{upper curvature} at $x$ in direction $\tau$. (See \cite{B}, p. 14.)

\begin{lem}
\label{lem:curv-max-chord} For any maximizing chord $c$ of a convex body, we
have
\[
\lc\tau(x)\geq\lgt(c)^{-1}%
\]
at each foot $x$ of $c$, and in each tangent direction $\tau$ at $x$.
\end{lem}

\begin{proof}
Proof. Let $c=xx^{\ast}$, and assume
\[
\lc\tau(x)<\lgt(c)^{-1}\text{;}%
\]
then there exists a sequence of points $\{x_{n}\}_{n=1}^{\infty}$ converging
to $x$, such that
\[
\Vert x_{n}-x^{\ast}\Vert>\lgt(c)\text{.}%
\]
But this obviously contradicts the hypothesis asking for $c$ to be maximizing.
\end{proof}

\begin{thm}
\label{thm:most-max-chord}For most convex bodies $K$ and any maximizing chord
$c$ of $K$,
\[
\lc\tau(x)\geq\lgt(c)^{-1}\ \ \rmand\ \ \uc\tau(x)=\infty
\]
at each foot $x$ of $c$, and in each tangent direction $\tau$ at $x$.
\end{thm}

\begin{proof}
By Theorem \ref{qthm:Klee-Gruber}, most convex bodies\ are smooth; so, one can
speak of tangent directions at boundary points. By Theorem \ref{qthm:TZ}, for
most convex body $K$ and any point $x\in\bd K$, we have
\[
\lc\tau(x)=0\ \ \rmor\ \ \uc\tau(x)=\infty
\]
in each tangent direction $\tau$.

Since, by Lemma \ref{lem:curv-max-chord}, we have
\[
\lc\tau(x)\neq0
\]
for every endpoint $x$ of a maximizing chord, and every tangent direction
$\tau$, the theorem follows.
\end{proof}

A chord $c$ which is longest among all chords of $C\in\bds$ is called a
\emph{metric diameter} of $C$. The next result strengthens Theorem
\ref{thm:most-max-chord} in the case of the metric diameter and improves
Theorem 11 in \cite{Zamfi-diam}.

\begin{thm}
\label{thm:metric-diam}Most convex bodies\ admit a single metric diameter
$c$,
\[
\lc\tau(x)=\lgt(c)^{-1}\ \ \rmand\ \ \uc\tau(x)=\infty
\]
at each foot $x$ of $c$, and in each tangent direction $\tau$ at $x$.
\end{thm}


\nc{\mcH}{\mathcal{K}^{\prime}}

\begin{proof}
A direction or a line-segment or a hyperplane will be called
\textit{horizontal}, respectively \textit{vertical}, if it is parallel,
respectively orthogonal, to a fixed hyperplane.

By Theorem 11 in \cite{Zamfi-diam}, most convex bodies have a single metric
diameter. As the set of all convex bodies having a horizontal diameter is
obviously nowhere dense, the space $\mcH$ of all convex bodies with a single
non-horizontal diameter is residual in $\bds$, and we apply Lemma \ref{lem:Z}
to obtain generic results in $\bds$, working in $\mcH$.

Let $xx^{\ast}$ be the metric diameter of $C\in\mcH$, such that $x$ is above
and $x^{\ast}$ below any horizontal hyperplane cutting $xx^{\ast}$, and let
the direction $\tau$ be orthogonal to $\overline{xx^{\ast}}$. Take the points
$x_{n}^{\ast}\in xx^{\ast}$ such that $\Vert x^{\ast}-x_{n}^{\ast}\Vert
=1/n$\ \ $(n=1,2,3,...)$, and consider the half-plane $\Pi$ with $xx^{\ast}$
on its relative boundary and $x+\tau\in\Pi.$

Let $A_{n}(\tau)\subset\Pi$ be the arc starting at $x$, of length $1/n$, of
the circle of centre $x_{n}^{\ast}$ passing through $x$. The radius is
$\diam
C-1/n$.

Let us say that $C\in\mcH$ has the $(n)$-property if for its metric diameter
$xx^{\ast}$ and for some direction $\tau$ orthogonal to $xx^{\ast}$,
$A_{n}(\tau)$ does not meet $\inn C$.

We prove that the set $\mcH_{n}$ of those $C\in\mcH$ which enjoy the
$(n)$-property is nowhere dense in $\mcH$.

First, it is easily seen that each $\mcH_{n}$ is closed in $\mcH$. Then, let
$C\in\mcH$. Approximate it by a polytope $P$ having as metric diameter
$xx^{\ast}$. Choose $\varepsilon>0$ very small (compared with $1/n$). Consider
the $(d-1)$-sphere $S$ with $\left\langle S\right\rangle $ orthogonal to
$xx^{\ast}$, having its centre on $xx^{\ast}$, lying between $C_{x_{n}^{\ast
}x}$ and $C_{x^{\ast}x}$, and satisfiying $\diam S=\varepsilon$. Let the
polytope $P^{\prime\prime}$ approximate $\conv S$ in $\left\langle
S\right\rangle $, with $d_{PH}\left(  P^{\prime\prime},S\right)  $ much
smaller than $\varepsilon$.

Then $P^{\prime}=\conv(P\cup P^{\prime\prime})$ has not the $(n)$-property,
whence $\mcH_{n}$ is nowhere dense. In conclusion, most $C\in\mcH$ have the
$(n)$-property for no natural number $n$. This means that for every tangent
direction $\tau$ at $x$,%
\[
\ucr\tau(x)>\diam C-1/n
\]
for infinitely many $n$'s, yielding $\ucr\tau(x)=\diam C.$

Analogously, $\ucr\tau(x^{\ast})=\diam C.$

By Theorem 1 in [8], for most $C\in\bds$, at every point $z\in\bd C$ and for
every direction $\tau$ at $z$, $\lcr\tau(z)=0$ or $\ucr\tau(z)=\infty.$ It
follows that, at the endpoints $x,x^{\ast}$ of the unique metric diameter and
for any direction $\tau$, $\lcr\tau(x)=\lcr\tau(x^{\ast})=0$.
\end{proof}

The above theorems describe the curvature at the feet of maximizing chords.
However, as shown by Proposition \ref{prp:countable}, maximizing chords are
rare among double normals. Concerning typical double normals we have the
following result.

\begin{thm}
For most $K\in\bds$ and most $x\in\BF(K)$, in any tangent direction $\tau$,
$\uc\tau(x)=\infty$.
\end{thm}

\begin{proof}
Rephrasing the second point of Theorem \ref{qthm:TZ}, we get that for most
$K\in\bds$ the set%
\[
\mathcal{I}=\set(:x\in\BF(K)|\forall\tau\in\TUS_{x}K,\uc\tau(x)=\infty:)
\]
contains a dense $G_{\delta}$ set in $\bd K$. A closer look at the proof in
the original paper \cite{Zamfi-curv II} shows that indeed $\mathcal{I}$
\emph{is} a $G_{\delta}$ set. Thus $\mathcal{I}\cap\BF(K)$ is a $G_{\delta}$
set in $\BF(K)$, which contains, by Theorem \ref{thm:most-max-chord}, all the
feet of maximizing chords of $K$. Now, by Proposition \ref{prp:dense}, the set
of those feet is dense in $\BF(K)$, whence the conclusion.
\end{proof}

\begin{rem}
\label{rmk:cfdn}We still ignore, for typical convex bodies, the behaviour of
the lower curvature at the feet of (most) double normals. The existence of
double normals with finite upper curvature at a foot is also unknown; however,
these curvatures cannot be finite at both feet of the same double normal (see
Theorem 4.1 in \cite{AZ1}).
\end{rem}

Let us consider a typical convex body among those that admit a given
line-segment as double normal.

\begin{thm}
\label{thm:fixed-d-normal}For most convex bodies\ admitting the double normal
$c$,
\[
\lc\tau(x)=0\ \ \rmand\ \ \uc\tau(x)=\infty
\]
at each foot $x$ of $c$, and in each tangent direction $\tau$ at $x$.
\end{thm}

\nc{\mcD}{\mathcal{K}^{\prime\prime}}

Theorem \ref{thm:fixed-d-normal} shows that the curvature behaviour at the
endpoints of $c$ coincides with the curvature behavior at most points. (See
\cite{Zamfi-curv II} for the latter result in $\mathcal{K}$; the result is
also valid in the space $\mcD$ defined below, and the proof parallels that for
$\mathcal{K}$.)

\begin{proof}
Let $\mcD$ be the Baire space of all convex bodies\ admitting $c$ as a double
normal. We may assume that $c=xx^{\ast}$ is vertical, with $x$ above $x^{\ast
}$.

Following the same steps as in the proofs of Klee \cite{Klee-B} or Gruber
\cite{Grub-B}, one can show that most $C\in\mcD$ are smooth (boundary of class
$\mathcal{C}^{1}$). This justifies the use of \textquotedblleft tangent
directions\textquotedblright\ at $x$.

Let the direction $\tau$ be orthogonal to $\overline{xx^{\ast}}$. Consider the
points $x_{n}\in\overline{xx^{\ast}}$, $x_{n}^{\prime}\in xx^{\ast}$, such
that $x\notin x^{\ast}x_{n}$ and $\Vert x-x_{n}\Vert=\Vert x-x_{n}^{\prime
}\Vert^{-1}=n.$ Take the half-plane $\Pi$ with $xx^{\ast}$ on its boundary and
$x+\tau\in\Pi.$

Let $A_{n}(\tau)\subset\Pi$, $A_{n}^{\prime}(\tau)\subset\Pi$ be the arcs
starting in $x$, of length $1/n$, of the circle of centre $x_{n}$,
respectively $x_{n}^{\prime}$, passing through $x$. The radii are $n$ and
$1/n$, respectively.

We now say that $C\in\mcD$ has the $(n)$-property if, for some horizontal direction
$\tau$, $A_{n}(\tau)\cap\inn C=\emptyset$ or
$A_{n}^{\prime}(\tau)\subset C.$

We prove that the set $\mcD_{n}$ of those $C\in\mcD$ which enjoy the
$(n)$-property is nowhere dense in $\mcD$.

Again, it is easily checked that each $\mcD_{n}$ is closed in $\mcD$.
Approximate $C\in\mcD$ by a polytope $P$ with vertices $x$, $x^{\ast}$ such
that $\bd P$ has no horizontal direction at $x$ or $x^{\ast}$. We now use the
polytope $P^{\prime}$ constructed in the proof of Theorem
\ref{thm:metric-diam}. This polytope has not the $(n)$-property, whence
$\mcD_{n}$ is nowhere dense. Hence, most $C\in\mcD$ have the $(n)$-property
for no natural number $n$. Thus, for every tangent direction $\tau$ at $x$,%
\[
\ucr\tau(x)>n\ \ \rmand\ \ \lcr\tau(x)<1/n
\]
for infinitely many $n$'s, i.e. $\ucr\tau(x)=\infty$ and $\lcr\tau(x)=0$.

Analogously, $\ucr\tau(x^{\ast})=\infty$ and $\lcr\tau(x^{\ast})=0$.
\end{proof}

\begin{ack}
A. Rivi\`{e}re and J. Rouyer thankfully acknowledge T. Zamfirescu's hospitality.

J. Rouyer thankfully acknowledges financial support from the\emph{ Centre
Francophone de Math\'{e}matique \`{a} l'IMAR}.

C. V\^\i lcu  acknowledges partial financial support from the grant of the Ministery of Research and Innovation, CNCSUEFISCDI,
project no. PN-III-P4-ID-PCE-2016-0019.

T. Zamfirescu thankfully acknowledges financial support by the High-end
Foreign Experts Recruitment Program of People's Republic of China.
\end{ack}

\medskip

Alain Rivi\`ere

\noindent{\footnotesize Laboratoire Ami\'enois de Math\'ematiques
Fondamentales et Appliqu\'ees\newline CNRS, UMR 7352\newline Facult\'e de
Sciences d'Amiens\newline80 039 Amiens Cedex 1, France.}

{\small \hfill Alain.Riviere@u-picardie.fr}

\medskip

Jo\"el Rouyer

\noindent{\footnotesize Simion Stoilow Institute of Mathematics of the
Roumanian Academy\newline Bucharest, Roumania }

{\small \hfill Joel.Rouyer@ymail.com}

\medskip

Costin V\^\i lcu

\noindent{\footnotesize Simion Stoilow Institute of Mathematics of the
Roumanian Academy\newline Bucharest, Roumania}

{\small \hfill Costin.Vilcu@imar.ro}

\medskip

Tudor Zamfirescu

\noindent{\footnotesize Fachbereich Mathematik, Universit\"at Dortmund\newline%
44221 Dortmund, Germany\newline and\newline Simion Stoilow Institute of
Mathematics of the Roumanian Academy\newline Bucharest, Roumania\newline
and\newline College of Mathematics and Information Science,\newline Hebei
Normal University,\newline050024 Shijiazhuang, P.R. China.}

{\small \hfill tudor.zamfirescu@mathematik.tu-dortmund.de}

\end{document}